\documentclass{amsart}
\usepackage{amsmath, amsthm, amsfonts, amssymb, color,url,float,ulem,float,cancel}
\usepackage[usenames,dvipsnames]{xcolor}
\usepackage{graphicx}

\theoremstyle{remark}
\newtheorem{rem}{Remark}

\theoremstyle{definition}
\newtheorem{def1}{Definition}
\def\Zi{\mathbb{Z}[i]}
\def\mfm{\mathfrak{m}}
\def\mfn{\mathfrak{n}}
\def\mfk{\mathfrak{K}}

\theoremstyle{plain}
\newtheorem{thm}{Theorem}

\newtheorem{lem}[thm]{Lemma}
\newtheorem{cor}[thm]{Corollary}
\newtheorem{conj}{Conjecture}
\newtheorem{conj2}{Conjecture}

\title{On the Distribution of the Greatest Common Divisor of Gaussian Integers}

\begin{document}
\author{Tai-Danae Bradley}
\address{B.S. Program in Mathematics\\
The City College of New York\\
Convent Ave. at 138th St., NY 10031}
\email[Tai-Danae Bradley]{tai.danae@gmail.com}
\author{Yin Choi Cheng}
\email[Yin Choi Cheng]{cycsano@hotmail.com}
\author{Yan Fei Luo}
\email[Yan Fei Luo]{fay.or.flymorning@gmail.com}

\keywords{Gaussian integer, gcd, moment, Dedekind zeta function}
\subjclass[2010]{11N37, 11A05, 11K65, 60E05}
\maketitle

\begin{abstract} 
For a pair of random Gaussian integers chosen uniformly and independently from the set of Gaussian integers of norm $x$ or less as $x$ goes to infinity, we find asymptotics for the average norm of their greatest common divisor, with explicit error terms. We also present results for higher moments along with computational data which support the results for the second, third, fourth, and fifth moments. The analogous question for integers is studied by Diaconis and Erd\"os. 
\end{abstract}

\section{Introduction}

In this paper, we study questions related to the size of the greatest common divisor of pairs of randomly chosen Gaussian integers. In particular, in Theorem \ref{thm:probthm}, we first calculate the probability that a pair of random Gaussian integers, chosen uniformly and independently from the set of all Gaussian integers with norm $x$ or less, has greatest common divisor $\pm \kappa$ or $\pm i\kappa$ for a fixed Gaussian integer $\kappa$. The main term for this probability in the case where $\kappa$=1 was first given by Collins and Johnson \cite[Theorem 8]{relprime}. We refine their results by providing the expression for the more general case in addition to giving an explicit error term for all cases. In Theorem \ref{thm:expthm} we derive the expected norm of  the greatest common divisor between a pair of Gaussian integers with norm $x$ or less. Finally, in Theorem \ref{thm:hmom} and Conjecture \ref{conj}, we present an expression for higher moments of the norm of the greatest common divisor between a pair of Gaussian integers with norm $x$ or less. We expect our results to generalize to principal ideal domains without too much difficulty. More generally, our results should hold for the ring of integers in an algebraic number field, though our techniques will need to be modified to deal with class number greater than one and infinite unit group. We expect the ideas in the recent preprint of \cite{Gia} could help address this question and would be an interesting direction to explore further. Of further interest are function field analogues. Some interesting results in this direction may be found in \cite{Gia2}.

Similar questions have also been studied for the case of rational integers. Originally, Mertens \cite{Mertens} proved in 1874 that the probability that a pair of rational integers chosen uniformly and independently at random from $\{1,2,\ldots,x\}$ are relatively prime is asymptotic to $1/\zeta(2)$, as $x$ tends to infinity, where $\zeta$ is the Riemann zeta function. In 1956, Christopher \cite[Theorem 1]{Chris} generalized Mertens' result by finding the probability that two integers have  greatest common divisor $k$ for a fixed $k$ larger than 1. An asymptotic expression for the moments of the greatest common divisor was first derived by Ces\`{a}ro in 1885 \cite{Cesaro}, and Diaconis and Erd\"os later extended his work by explicitly calculating the error term \cite[Theorem 2]{Erd}. In particular, the expected value for the greatest common divisor between a pair of random integers chosen independently and uniformly from the set $\{1,2,\ldots,x\}$ is 
 \begin{equation} \label{eq:Erd1}
\frac{1}{\zeta(2)}\log{x} + O(1)
\end{equation}
while the $n^{\text{th}}$ moment is given by
\begin{align}\label{eq:Erd2}
\frac{x^{n-1}}{n+1}\left\{ \frac{2\zeta(n)}{\zeta(n+1)}-1  \right\} +O\left(x^{n-2}\log{x}\right) \qquad \mbox{ for $n\geq 2$}. 
\end{align}

The goal of the present paper is to show that equation (\ref{eq:Erd1}) has an analogous counterpart in the ring of Gaussian integers as stated in Theorem \ref{thm:expthm} at the end of this section. Further, we show that equation (\ref{eq:Erd2}) also has an analogous form as presented in Theorem \ref{thm:hmom} and Conjecture \ref{conj}. Before proceeding, we first give the following preliminary definitions and remark.

\begin{def1} The norm of a Gaussian integer $\alpha=a+bi$ for rational integers $a$ and $b$ is defined by $N(\alpha)=a^2+b^2$. 
\end{def1}
Most of our results will be in terms of the norms of Gaussian integers and not the integers themselves. 

\begin{rem} Given two Gaussian integers $\eta$ and $\mu$, a greatest common divisor, denoted $(\eta,\mu)$, is defined to be a Gaussian integer $\kappa$ such that $\kappa$ is a divisor of both $\eta$ and $\mu$, and if there is any other common factor between $\eta$ and $\mu$, then it must also be a factor of $\kappa$. From this definition, it becomes clear that $(\eta,\mu)$ is unique only up to its associates. In other words, $(\eta,\mu)=\kappa,-\kappa,i\kappa, \mbox{ and }-i\kappa$. Our calculations, however, will be performed via ideals\ for reasons that will soon become apparent. For a Gaussian integer $\eta$, $\mathfrak{n}$ is the ideal such that $\mathfrak{n}=(\eta)=(-\eta)=(i\eta)=(-i\eta)$, and the norm of $\mathfrak{n}$ is defined by $N(\mathfrak{n})=N(\eta)$. Accordingly, the definition of the greatest common divisor for a pair of ideals is: 
\end{rem}

\begin{def1}[Greatest Common Divisor of Two Ideals] For a ring $R$, let $\mathfrak{n},\mathfrak{m}\subset R$ be ideals. The greatest common divisor $(\mathfrak{n},\mathfrak{m})$ is defined to be the ideal $\mathfrak{K}\subset R$ which satisfies the following:
\begin{enumerate}
	\item $\mathfrak{n}\subset\mathfrak{K}$ and $\mathfrak{m}\subset\mathfrak{K}$
	\item if there exists some ideal $\mathfrak{a}\subset R$ such that $\mathfrak{n}\subset\mathfrak{a}$ and $\mathfrak{m}\subset\mathfrak{a}$, then $\mathfrak{K}\subset\mathfrak{a}$.
\end{enumerate}
\end{def1}
\noindent In other words, $(\mathfrak{n},\mathfrak{m})$ is the smallest ideal that contains all the elements of both $\mathfrak{n}$ and $\mathfrak{m}$. When applied  to the ring of Gaussian integers, a Dedekind domain, it is clear that $(\mathfrak{n},\mathfrak{m})$ is unique.

\begin{def1}[The Dedekind Zeta Function of $\mathbb{Q}(i)$] For the number field $\mathbb{Q}(i)$, the complex-valued Dedekind zeta function is defined for Re$(s)>1$ by
\begin{equation*}
\zeta_{\mathbb{Q}(i)}(s) \quad = \quad \sum_{\mathfrak{a}\subset \mathbb{Z}[i]}{\frac{1}{N(\mathfrak{a})^s}}  \quad = \quad
\frac{1}{4}\sum_{\underset{(a,b)\neq(0,0)}{{(a,b)\in\mathbb{Z}}}}{\frac{1}{(a^2+b^2)^s}}
\end{equation*} 
where the first summation is over the nonzero ideals $\mathfrak{a}$ of the ring of Gaussian integers $\mathbb{Z}[i]$.  \end{def1}
In order to find the expression for the expected norm of a greatest common divisor between a pair of Gaussian integers of norm $x$ or less, we will first derive the necessary probability distribution function of Theorem \ref{thm:probthm}:
\begin{thm}\label{thm:probthm} Let $\mathfrak{n}$ and $\mathfrak{m}$ be nonzero ideals chosen independently and uniformly at random from the set of ideals in $\Zi$ with norm $x$ or less. The probability that $(\mathfrak{n},\mathfrak{m})=\mathfrak{K}$ is
\begin{equation*} 
\frac{1}{\zeta_{\mathbb{Q}(i)}(2)N(\mathfrak{K})^2} + O\left(\frac{1}{x^{2/3}N(\mathfrak{K})^{4/3}}\right).
\end{equation*}
\end{thm}
This probability will allow us to calculate the expected norm of the greatest common divisor between a pair of ideals:

\begin{thm}\label{thm:expthm} Let $\mathfrak{n}$ and $\mathfrak{m}$ be nonzero ideals chosen independently and uniformly at random from the set of ideals in $\Zi$ with norm $x$ or less. The expected norm of the greatest common divisor of $\mathfrak{n}$ and $\mfm$ is
 \begin{equation*} \label{eq:main}
 \frac{\pi }{4\zeta_{\mathbb{Q}(i)}(2)}\log{x} + O(1).\\
 \end{equation*}
\end{thm}

We will then prove the following result regarding the $n^{\text{th}}$ moment for $n>2$:

\begin{thm}\label{thm:hmom} 
Let $\mathfrak{n}$ and $\mathfrak{m}$ be nonzero ideals chosen independently and uniformly at random from the set of ideals in $\Zi$ with norm $x$ or less. For $n> 2$, there exists a constant $c_n\in\mathbb{R}$ such that
\begin{align*}
E_x\{N(\mfn,\mfm)^n\}\sim  c_nx^{n-1} 
\end{align*}
where $E_x\{N(\mfn,\mfm)^n\}$ denotes the $n^\text{th}$ moment of the norm of the greatest common divisor of $\mfn$ and $\mfm$.
\end{thm}

Lastly, we will present numerical data which provide strong evidence for the following conjecture regarding the constant of Theorem \ref{thm:hmom} for all $n\geq 2$:
\begin{conj}\label{conj}
For $n\geq 2$, 
 \begin{equation*} 
E_x\{N(\mfn,\mfm)^n\}\sim \frac{4 }{\pi(n+1)}\left\{\frac{2\zeta_{\mathbb{Q}(i)}(n)}{\zeta_{\mathbb{Q}(i)}(n+1)}-1\right\}x^{n-1}. 
 \end{equation*}
\end{conj}

The proof of Theorem \ref{thm:probthm} will be given in Section \ref{sec:sec2} and that of Theorem \ref{thm:expthm} will be given in Section \ref{sec:sec3}. Finally, in Section \ref{sec:sec4}, we prove Theorem \ref{thm:hmom} and present Conjecture \ref{conj} along with computational data which support the conjecture for the second, third, fourth, and fifth moments.

\section{Probability Distribution Function}\label{sec:sec2}

Before deriving the expression for the probability of Theorem \ref{thm:probthm}, we first define the following two functions:

\begin{def1}[The M\"obius Function]
For an ideal $\mathfrak{n}$, the M\"obius Function $\mu(\mathfrak{n})$ is defined by: 
$$ \mu(\mathfrak{n}) = 
	\begin{cases}
		1,		&\text{if $\mathfrak{n}$ = $(1)$}; \\
		(-1)^t, 	&\text{if $\hspace{2pt} \mathfrak{n} = \mathfrak{p}_1 \mathfrak{p}_2 \cdots \mathfrak{p}_t$ for distinct prime ideals $\mathfrak{p}_i$};\\
		0, 		&\text{otherwise}.
	\end{cases}$$
We will use the following identity
\begin{align}\label{eq:mob}
\sum_{\mathfrak{d}|\mathfrak{n}}{\mu(\mathfrak{d})} =
	\begin{cases}	
		1, &\text{if $\mathfrak{n}=(1)$};\\
		0, &\text{if $\mathfrak{n}\neq (1)$}
	\end{cases}
	\end{align} 
as well as the generating function
\begin{equation*} \label{eq:genfcn}
\sum_{\mathfrak{n}\subset \Zi}{\frac{\mu(\mathfrak{n})}{N(\mathfrak{n})^s}} = \frac{1}{\zeta_{\mathbb{Q}(i)}(s)} \parbox[c]{6cm}{\qquad \qquad \qquad for Re$(s)>$1.}
\end{equation*}
\end{def1}

\begin{def1}[The Sum of Two Squares Function]   For $n \in \mathbb{Z}$, let the sum of two squares function $r(n,2)$ represent the number of ways that $n$ can be expressed as a sum of two squares. Thus,  
\begin{equation*}
r(n,2) = \frac{1}{4}\#\{\mathfrak{a} \subset \mathbb{Z}[i] : N(\mathfrak{a})=n\}.
\end{equation*}
We will need the result of Sierpi\'{n}ski \cite{sqfcn} (for a statement in English see \cite[equation (1)]{Schin}) 
\begin{equation} \label{eq:rsum}
\sum_{n=1}^{x}{r(n,2)} = \pi x + O(x^\frac{1}{3}).
\end{equation}
The error term $O(x^{1/3})$ has been improved by Huxley \cite{Hux} to $O(x^{\frac{131}{416} + \epsilon})$, but the former is sufficient for our purposes. We shall also use the following \cite{Sierp} 
\begin{equation} \label{eq:rnsum}
\sum_{n=1}^{x}{\frac{r(n,2)}{n}} =  \pi(S+ \log{x}) + O(x^{-1/2})
\end{equation}
where $S$ denotes Sierpi\'{n}ski's constant $S\approx2.58/\pi$. This also has the alternate expressions \cite[p. 123]{K} 
\begin{equation*}
S= \frac{1}{\pi} \lim_{z \to \infty}\left(4\zeta(z)\beta(z) -\frac{\pi}{z-1}\right) =      \gamma+ \frac{\beta '(1)}{\beta(1)} \end{equation*}
where
$\beta(z)$ is the Dirichlet beta function and $\gamma$ is the Euler-Mascheroni constant.

\end{def1}
With these functions at hand, we may now proceed to calculate the desired probability. To do so, we will need two preliminary results. The first is the total number of pairs of ideals generated by Gaussian integers with norm at most $x$. The second result is the number of those pairs which have greatest common divisor $\mfk$. The expressions for each of these are derived in the following two lemmas. 
\begin{lem}\label{lem:lem1}
The total number of pairs of nonzero ideals $\mathfrak{n}$ and $\mathfrak{m}$ in $\Zi$ with norm $x$ or less is
\begin{align*}
\frac{\pi^2 x^2}{16}  + O(x^{4/3}).
\end{align*}
\end{lem}

\begin{proof}
Let $\mfn$ and $\mfm$ be nonzero ideals. Then
\begin{align*}
\#\{\mathfrak{n}, \mathfrak{m} \subset \mathbb{Z}[i]^2 : N(\mathfrak{n}), N(\mathfrak{m}) \leq x\} 
&= 	\sum_{\underset {N(\mfn) \leq x}{\mathfrak{n}\subset \mathbb{Z}[i]}} \, \, \, \sum_{\underset {N(\mfm) \leq x}{\mathfrak{m}\subset \mathbb{Z}[i]}}{1},
\end{align*}
and we may rewrite this as
\begin{align*}
& \frac{1}{16} \sum_{N(\mathfrak{n})=1}^{ \lfloor x\rfloor} {r\left(N(\mathfrak{n}),2\right)} \sum_{N(\mathfrak{m})=1}^{  \lfloor x\rfloor} {r\left(N(\mathfrak{m}),2\right)} 
\end{align*}
which by equation (\ref{eq:rsum}) equals
\begin{align}
&= \frac{1}{16} \left[ \pi x + O(\lfloor x \rfloor^{1/3}) \right]^2. \nonumber 
\end{align}
Further, since $O(\lfloor x \rfloor)=O(x)$ we may expand $\left[\pi x + O(x^{1/3})\right]^2$ and obtain $\pi^2 x^2 + 2\pi x \cdot O(x^{1/3}) + O(x^{2/3})$ which reduces to $\pi^2 x^2 + O(x^{4/3}).$ Thus, the total number of $\mfn$ and $\mfm$ with norm at most $x$ is
\begin{align*}
\frac{\pi^2 x^2 }{16} + O(x^{4/3}).
\end{align*}
\end{proof}

\begin{lem}\label{lem:lem2} 
The total number of pairs of nonzero ideals $\mathfrak{n}$ and $\mathfrak{m}$ in $\Zi$ with norm $x$ or less having greatest common divisor $\mfk$ is
\begin{equation*}
\frac{\pi^2x^2}{16\zeta_{\mathbb{Q}(i)}(2)k^2} + O\left(\frac{x^{4/3}}{k^{4/3}}\right)
\end{equation*} 
where $k=N(\mfk)$.
\end{lem}

\begin{proof}
Let $\mathfrak{n}$ and $\mathfrak{m}$ be nonzero ideals. The number of pairs of $\mfn$ and $\mfm$ with norm $x$ or less which are relatively prime is 
\begin{align*}
\#\{\mathfrak{n}, \mathfrak{m} \subset \mathbb{Z}[i]^2 :  N&(\mathfrak{n}), N(\mathfrak{m}) \leq  x \mbox{ and } (\mathfrak{n}, \mathfrak{m}) = (1)\} \nonumber \\[5pt]
&= 	\sum_{\underset { N(\mfn) \leq x}{\mathfrak{n}\subset \mathbb{Z}[i]}} \, \, \, \sum_{\underset{(\mfn,\mfm)=(1)}{\underset { N(\mfm) \leq x}{\mathfrak{m}\subset \mathbb{Z}[i]}}}{1} \nonumber \\[5pt] 	
&= 	\sum_{\underset { N(\mfn) \leq x}{\mathfrak{n}\subset \mathbb{Z}[i]}} \, \, \,\sum_{\underset { N(\mfm) \leq x}{\mathfrak{m}\subset \mathbb{Z}[i]}} \, \, \,\sum_{\underset{ \mathfrak{d}|(\mfn,\mfm)}{\mathfrak{d}\subset\Zi}}{\mu(\mathfrak{d})} \nonumber 
\end{align*}
where in the last line we have used identity (\ref{eq:mob}). Reindexing with $\mathfrak{n} = \mathfrak{d}\mathfrak{n}'$ and $\mathfrak{m} = \mathfrak{d}\mathfrak{m}'$ where the norms of $\mfn'$ and $\mfm'$ range from 1 to $x/N(\mathfrak{d})$, we may rewrite this as
\begin{align*}
& \sum_{\underset{  N(\mathfrak{d}) \leq x}{\mathfrak{d}\subset\Zi}}{\mu(\mathfrak{d})} \sum_{\underset { N(\mathfrak{n}') \leq \frac{x}{N(\mathfrak{d})}}{\mathfrak{n}'\subset \mathbb{Z}[i]}}  \,\sum_{\underset {N(\mathfrak{m}') \leq \frac{x}{N(\mathfrak{d})}}{\mathfrak{m}'\subset \mathbb{Z}[i]}}{1}   \\[6pt]
&= \frac{1}{16}\sum_{\underset{ N(\mathfrak{d}) \leq x}{\mathfrak{d}\subset\Zi}}{\mu(\mathfrak{d})} 
	\sum_{N\mathfrak{n}'=1}^{\left\lfloor \frac{x}{N\mathfrak{d}}\right\rfloor}{r\left(N(\mathfrak{n}'),2\right)} 
	\sum_{N\mathfrak{m}'=1}^{\left\lfloor \frac{x}{N\mathfrak{d}}\right\rfloor}{r\left(N(\mathfrak{m}'),2\right)}. 
\end{align*}
As in Lemma \ref{lem:lem1}, this reduces to
\begin{align*}
\frac{1}{16}\sum_{\underset{ N(\mathfrak{d}) \leq x}{\mathfrak{d}\subset\Zi}}{\mu(\mathfrak{d})} \left[ \frac{\pi^2x^2}{N(\mathfrak{d})^2} + O \left(\frac{x}{N(\mathfrak{d})}\right)^{4/3} \right].
\end{align*}

 We then distribute the summation to obtain

\begin{align} 
&  \frac{\pi^2x^2}{16}  \sum_{\underset{ N(\mathfrak{d}) \leq x}{\mathfrak{d}\subset\Zi}}{\frac{\mu({\mathfrak{d})}}{N(\mathfrak{d})^2}} 	\quad + \quad O\left( \sum_{\underset{  N(\mathfrak{d}) \leq x}{\mathfrak{d}\subset\Zi}}{\left(\frac{x}{N(\mathfrak{d})}\right)^{4/3}}\right). \label{eq:mobius}
\end{align}
To evaluate the main term, we call on the generating function $\sum_{\mfn\subset\Zi}{\frac{\mu(\mathfrak{n})}{N(\mathfrak{n})^s}} = 1/\zeta_{\mathbb{Q}(i)}(s)$ for Re$(s)>1$, to see that
\begin{align*}
\sum_{\underset{ N(\mathfrak{d}) \leq x}{\mathfrak{d}\subset\Zi}}{\frac{\mu({\mathfrak{d})}}{N(\mathfrak{d})^2}}  
&= \frac{1}{\zeta_{\mathbb{Q}(i)}(2)} - \sum_{n=x+1}^{\infty}\sum_{\underset{ N(\mathfrak{d})=n}{\mathfrak{d}\subset\Zi}}{\frac{\mu(\mathfrak{d})}{N(\mathfrak{d})^2}}
\end{align*}
which implies
\begin{align*}
\bigg| \frac{1}{\zeta_{\mathbb{Q}(i)}(2)}   - \sum_{\underset{  N(\mathfrak{d}) \leq x}{\mathfrak{d}\subset\Zi}}{\frac{\mu({\mathfrak{d})}}{N(\mathfrak{d})^2}} \bigg| \quad \leq \quad \sum_{n=x+1}^{\infty}{\frac{1}{n^2}} \sum_{\underset{ N(\mathfrak{d})=n}{\mathfrak{d}\subset\Zi}}{1} \quad = \quad \frac{1}{4}\sum_{n=x+1}^{\infty}{\frac{r(n,2)}{n^2}}.
\end{align*}
Now we note that $r(n,2)\leq4\sigma_0(n)=o(n^\epsilon)$ for all $\epsilon>0$, where $\sigma_0$ represents the number of divisors of $n$. Thus $\frac{1}{4}\sum_{n=x+1}^{\infty}{\frac{r(n,2)}{n^2}} \leq \sum_{n=x+1}^{\infty}{\frac{o(n^\epsilon)}{n^2}}=o(x^{\epsilon-1})$ and so
\begin{align*} 
\bigg| \frac{1}{\zeta_{\mathbb{Q}(i)}(2)}   - \sum_{\underset{  N(\mathfrak{d}) \leq x}{\mathfrak{d}\subset\Zi}}{\frac{\mu({\mathfrak{d})}}{N(\mathfrak{d})^2}} \bigg| \leq o(x^{\epsilon-1}) \qquad \mbox{or} \qquad
\sum_{\underset{ N(\mathfrak{d}) \leq x}{\mathfrak{d}\subset\Zi}}{\frac{\mu({\mathfrak{d})}}{N(\mathfrak{d})^2}}  = \frac{1}{\zeta_{\mathbb{Q}(i)}(2)} + o(x^{\epsilon-1}). 
\end{align*}
For the error term of (\ref{eq:mobius}), we have
\begin{align*}
\sum_{\underset{  N(\mathfrak{d}) \leq x}{\mathfrak{d}\subset\Zi}}{\left(\frac{1}{N(\mathfrak{d})}\right)^{4/3}} 
\quad &= \quad   \sum_{n=1}^{x}{\frac{1}{n^{4/3}}} \sum_{\underset{N(\mathfrak{d})=n}{\mathfrak{d}\subset\Zi}}{1}
\quad = \quad \frac{1}{4} \sum_{n=1}^{x}{\frac{r(n,2)}{n^{4/3}}}
\end{align*}
and again use the bound $r(n,2)\leq o(n^\epsilon)$ to see that $\frac{1}{4} \sum_{n=1}^{x}{\frac{r(n,2)}{n^{4/3}}} \leq  \sum_{n=1}^{x}{o(n^{\epsilon-4/3})}$ which equals $o(x^{\epsilon-1/3})+o(1).$ From this it is clear that $O\left(x^{4/3} \sum_{n=1}^{x}{\frac{r(n,2)}{n^{4/3}}}\right) = O(o(x^{4/3})) = O(x^{4/3})$. Thus (\ref{eq:mobius}) becomes 
\begin{align*}
\frac{\pi^2x^2}{16\zeta_{\mathbb{Q}(i)}(2)} + o(x^{\epsilon-1}) + O(x^{4/3})
\end{align*}
which allows us to conclude 
\begin{equation*}
\# \{\mathfrak{n}, \mathfrak{m} \subset \mathbb{Z}[i]^2 : N(\mathfrak{n}), N(\mathfrak{m}) \leq x \mbox{ and } (\mathfrak{n}, \mathfrak{m}) = (1)\} = \frac{\pi^2x^2}{16\zeta_{\mathbb{Q}(i)}(2)} + O(x^{4/3}).
\end{equation*}
Having counted the number of relatively prime $\mfn$ and $\mfm$ within a given norm, we can now reindex to obtain the number of them which have $(\mfn,\mfm)=\mfk$. Letting $\mathfrak{n}=\mathfrak{n}'\mathfrak{K}$ and $\mathfrak{m}=\mathfrak{m}'\mathfrak{K}$, we see that $\mathfrak{n}'$ and $\mathfrak{m}'$ are relatively prime if and only if $\mfn$ and $\mfm$ have $\mfk$ as their greatest common divisor. Hence, the number of relatively prime pairs $\mathfrak{n}'$ and $\mathfrak{m}'$ with norm $y$ or less must be equivalent to the number of pairs $\mathfrak{n}$ and $\mathfrak{m}$, with norm $yk$ or less (where $k=N(\mathfrak{K})$), having greatest common divisor $\mfk$. Thus,

\begin{align}
\#\{\mathfrak{n}, \mathfrak{m} \subset \mathbb{Z}[i]^2 :   N&(\mathfrak{n}),  N(\mathfrak{m}) \leq x  \mbox{ and } (\mathfrak{n}, \mathfrak{m})  = \mathfrak{K}\}\nonumber \\[6pt]
 & = \#\{\mathfrak{n}',\mathfrak{m}' \subset \mathbb{Z}[i]^2 :  N(\mathfrak{n}'), N(\mathfrak{m}') \leq \frac{x}{k} \mbox{ and } (\mathfrak{n}',\mathfrak{m}')=(1)\} \nonumber \\[6pt]
&= \frac{\pi^2x^2}{16\zeta_{\mathbb{Q}(i)}(2)k^2} + O\left(\frac{x^{4/3}}{k^{4/3}}\right) . \nonumber
\end{align}
\end{proof}

At last, the probability that $\mfn$ and $\mfm$, having norm at most $x$, will have greatest common divisor $\mfk$ is defined to be the number of pairs of ideals of norm $x$ or less which have greatest common divisor $\mfk$ divided by the total number of pairs of ideals of norm $x$ or less. Thus, by Lemmas \ref{lem:lem1} and \ref{lem:lem2}, 
\begin{align}
P_x\{\mfn, \mfm \subset \mathbb{Z}[i]^2 :  N&(\mfn), N(\mfm) \leq x \mbox{ and } (\mfn,\mfm)=\mfk\} \label{eq:clean} \\ \nonumber
&=  \left[\frac{\pi^2 x^2}{16} + O(x^{4/3})\right]^{-1} \cdot \left[\frac{\pi^2x^2}{16\zeta_{\mathbb{Q}(i)}(2)k^2} + O\left(\frac{x^{4/3}}{k^{4/3}}\right)\right].\label{eq:clean}
\end{align}
We can rewrite $\left[\frac{\pi^2 x^2}{16} + O(x^{4/3})\right]^{-1}$ as $16\pi^{-2} x^{-2}\left[1+O(x^{-2/3})\right]^{-1}$ which is equal to $16\pi^{-2}x^{-2}\left[1+O(x^{-2/3})\right]$ since $\left[1+f(x)\right]^{-1}=1+O(f(x))$ for $f(x)$ tending towards 0 as $x$ approaches infinity.

Line (\ref{eq:clean}) then becomes
\begin{align*} 
&\pi^{-2}x^{-2}\left[1+O(x^{-2/3})\right] \cdot\left[\frac{\pi^2x^2}{\zeta_{\mathbb{Q}(i)}(2)k^2} + O\left(\frac{x^{4/3}}{k^{4/3}}\right)\right] \\[5pt]
&= \left[1+O(x^{-2/3})\right] \cdot \left[\frac{1}{\zeta_{\mathbb{Q}(i)}(2)k^2} + O\left(\frac{1}{x^{2/3}k^{4/3}}\right)\right] \\[5pt]
&= \frac{1}{\zeta_{\mathbb{Q}(i)}(2)k^2} + O\left(\frac{1}{x^{2/3}k^{4/3}}\right) + O\left(\frac{1}{x^{2/3}k^2}\right) + O\left(\frac{1}{x^{4/3}k^{4/3}}\right)
\end{align*}
or finally
\begin{align*}
 \frac{1}{\zeta_{\mathbb{Q}(i)}(2)k^2} \quad + \quad O\left(\frac{1}{x^{2/3}k^{4/3}}\right), 
\end{align*}
completing the proof of Theorem \ref{thm:probthm}. The following corollary is a direct consequence of Theorem \ref{thm:probthm} for the special case when $\mathfrak{K}=(1)$.

\begin{cor}\label{cor:cor} The probability that a pair of Gaussian integers with norm $x$ or less are relatively prime is
\begin{equation*}
\frac{1}{\zeta_{\mathbb{Q}(i)}(2)} + O\left(\frac{1}{x^{2/3}}\right).
\end{equation*}
\end{cor}

In effect, Corollary \ref{cor:cor} tells us that for $x$ large, the probability that two Gaussian integers are relatively prime is asymptotic to $[\zeta_{\mathbb{Q}(i)}(2)]^{-1}$ as $x$ tends towards infinity. This is in agreement with the work of Collins and Johnson who state the probability as $[\zeta_{\mathbb{Q}(i)}(2)]^{-1}=[\zeta(2)L(2,\chi)]^{-1}\approx0.6637$, where $L(2,\chi)$ is a Dirichlet L-series and $\chi$ the primitive Dirichlet character modulo 4.

\section{Expected Value}\label{sec:sec3}
Having derived the probability distribution function found in Theorem \ref{thm:probthm}, we are ready to find an expression for the expected value of our random variable, $N(\mfn, \mfm)=k$, where the norm of $\mfn$ and $\mfm$ ranges from 1 to $x$. To do this, we must express our probability in terms of $k$ as well. The modification is simple, however. Since the number of ideals with norm $k$ is equivalent to $r(k,2)/4$, the probability that the greatest common divisor of $\mathfrak{n}$ and $\mathfrak{m}$ has norm $k$ must be
\begin{align*} 
 P_x\{N(\mathfrak{n},\mathfrak{m})=k\}=\frac{r(k,2)}{4\zeta_{\mathbb{Q}(i)}(2)k^2}  +  O\left(\frac{r(k,2)}{x^{2/3}k^{4/3}}\right). 
\end{align*}
Then, by definition of expected value
\begin{align}
E_x\{N(\mathfrak{n},\mathfrak{m})\} 
&= \sum_{k=1}^{x}{k\cdot P_x\{N(\mathfrak{n},\mathfrak{m})=k\}} \nonumber \\[6pt]
&= \sum_{k=1}^{x}{k\cdot \left[ \frac{r(k,2)}{4\zeta_{\mathbb{Q}(i)}(2)k^2}  +   O\left(\frac{r(k,2)}{x^{2/3}k^{4/3}}\right) \right]} \nonumber \\[6pt]
&= \frac{1}{4\zeta_{\mathbb{Q}(i)}(2)} \sum_{k=1}^{x}{\frac{r(k,2)}{k}} \quad + \quad O\left(\frac{1}{x^{2/3}}\sum_{k=1}^{x}\frac{r(k,2)}{k^{1/3}}\right). \label{eq:byparts}  
\end{align}
Using Stieltjes integration by parts to evaluate the error term, we obtain
\begin{align*}
\sum_{k=1}^{x}\frac{r(k,2)}{k^{1/3}} &= x^{-1/3}\sum_{k=1}^{x}{r(k,2)} - 4 - \int_1^x \left( \pi k + O\left(k^{1/3}\right)\right)\left(-\frac{1}{3}k^{-4/3}\right) \,dk \\[5pt]
&= \frac{3\pi}{2} x^{2/3} + O(\log{x})
\end{align*}
which implies $O(x^{-2/3}\sum_{k=1}^{x}\frac{r(k,2)}{k^{1/3}})=O(1+x^{-2/3}\log{x})=O(1)$. The main term of (\ref{eq:byparts}) can be rewritten using Sierpi\'{n}ski's identity from equation (\ref{eq:rnsum}). Thus the expected value is equal to
 \begin{align*}
\frac{1}{4\zeta_{\mathbb{Q}(i)}(2)} \left[\pi (S + \log{x}) + O\left(\frac{1}{x^{1/2}}\right)\right] +O\left(1\right)
\end{align*}
or
\begin{align*} 
\frac{\pi}{4\zeta_{\mathbb{Q}(i)}(2)} \log{x}+ O(1).
\end{align*}
This completes the proof of Theorem \ref{thm:expthm}.

\section{Higher Moments}\label{sec:sec4}
At last, we show that there exists some constant $c_n\in\mathbb{R}$ such that the main term of the $n^{\text{th}}$ moment of $N(\mfn,\mfm)$ must be of the form $c_nx^{n-1}$ for $n> 2$. Let $N(\mfn),N(\mfm)\leq x$ with $(\mfn,\mfm)=\mfk$ and restrict $N(\mfk)$ to the interval $\left(\frac{x}{j+1},\frac{x}{j}\right]$. We may then write $\mfn=\mfn'\mfk$ and $\mfm=\mfm'\mfk$ where $(\mfn',\mfm')=(1)$. The restriction on the norm of $\mfk$ allows us to see that $N(\mfn'),N(\mfm')< x\cdot\frac{j+1}{x}$ which implies $N(\mfn'),N(\mfm')\leq j$. Now define
\begin{align*}
f(j) = \#\{(\mfn',\mfm')\subset \Zi^2:N(\mfn'),N(\mfm')\leq j \text{ and } (\mfn',\mfm')=(1)\} 
\end{align*}
for $j\in\mathbb{N}$. By Lemma \ref{lem:lem2},
\begin{align*}
f(j)&=\frac{\pi^2j^2}{16\zeta_{\mathbb{Q}(i)}(2)}+O(j^{4/3})\\[3pt]
&= O(j^2).
\end{align*}

Our reindexing above shows that this expression for $f(j)$ also gives us the number of pairs of ideals with norm $x$ or less having greatest common divisor $\mfk$ where $\frac{x}{j+1} < N(\mfk)\leq \frac{x}{j}$. Thus the $n^{\text{th}}$ moment of $N(\mfn,\mfm)$ is given by
\begin{align}\label{eq:hmom}
E_x\{N(\mfn,\mfm)^n\}=  \frac{1}{\left[\pi^2x^2/16+O(x^{4/3})\right]}\cdot \left[\sum_{j=1}^{x}{f(j)}\sum_{\underset{\frac{x}{j+1}< N(\mfk) \leq \frac{x}{j}}{\mfk\subset\Zi}}{N(\mfk)^n}\right].
\end{align}

We next turn our attention to the inner sum of (\ref{eq:hmom}). First note that $$\sum_{\underset{\frac{x}{j+1}< N(\mfk) \leq \frac{x}{j}}{\mfk\subset\Zi}}{N(\mfk)^n}=\frac{1}{4}\sum_{k=\lceil \frac{x}{j+1}\rceil}^{ \lfloor \frac{x}{j}\rfloor}{k^n r(k,2)}$$ where $k=N(\mfk)$. Then Stieltjes integration by parts yields

\begin{align*}
\frac{1}{4}\sum_{k=\lceil \frac{x}{j+1}\rceil}^{ \lfloor \frac{x}{j}\rfloor}{k^n r(k,2)} 
&= \lfloor x/j \rfloor^n\sum_{k=1}^{\lfloor \frac{x}{j}\rfloor}{r(k,2)} - \lceil \frac{x}{j+1}\rceil^n\sum_{k=1}^{\lceil \frac{x}{j+1}\rceil}{r(k,2)} - \int_{\lceil \frac{x}{j+1}\rceil}^{\lfloor \frac{x}{j}\rfloor} nt^{n-1}\left(\pi t +O(t^{1/3})\right) \,dt \\[3pt]
&= \frac{\pi}{4(n+1)}x^{n+1}\left[\frac{1}{j^{n+1}} - \frac{1}{(j+1)^{n+1}}\right] + O\left(\frac{x}{j}\right)^{n+1/3}.\\
\end{align*}

The numerator of $E_x\{N(\mfn,\mfm)^n\}$ is now equal to

\begin{align} \label{eq:numerator}
\frac{\pi}{4(n+1)}x^{n+1}\sum_{j=1}^{x}O(j^2) \left( \frac{(j+1)^{n+1}-j^{n+1}}{j^{n+1}(j+1)^{n+1}} \right)+ x^{n+1/3}\sum_{j=1}^{x}O(j^2)O\left( \frac{1}{j^{n+1/3}} \right).
\end{align}
The sum on the left is
\begin{align*}
\sum_{j=1}^{x}O(j^2)O\left(\frac{1}{j(j+1)^{n+1}}\right)=\sum_{j=1}^{x}O\left(\frac{1}{j^{-1}(j+1)^{n+1}}\right)
\end{align*}
which is bounded above by $\sum_{j=1}^{x}O(1/j^n)$. For $x$ tending toward infinity and $n\geq2$, this converges to some constant $c_n'\in\mathbb{R}$. A similar argument shows that the second sum of (\ref{eq:numerator}) is likewise convergent for $n>2$. We thus conclude that the main term of $E_x\{N(\mfn,\mfm)^n\}$ is of the form $c_n'x^{n+1}$.

At last, we divide this by the total number of pairs of ideals $\mfn,\mfm$ with norm at most $x$ to obtain the main term of the $n^\text{th}$ moment of $N(\mfn,\mfm)$ for $n>2$:
\begin{align*}
 \frac{ c_n'x^{n+1}}{\left[\pi^2x^2/16+O(x^{4/3})\right]}= c_nx^{n-1}\frac{1}{1+O(x^{-2/3})}
\end{align*}
where $c_n=16c_n'/\pi^2$. Since $[1+O(x^{-2/3})]^{-1}=[1+O(x^{-2/3})]$, it follows that for $x$ tending to infinity
\begin{align*}
E_x\{N(\mfn,\mfm)^n\}\sim c_nx^{n-1}.
\end{align*}

With this, we bring the proof of Theorem \ref{thm:hmom} to an end and close by restating our conjecture regarding the constant of $E_x\{N(\mfn,\mfm)^n\}$ for all $n\geq 2$. We also include numerical evidence below which provides support for the conjecture in the cases when $n=2,3,4$ and 5.

\begin{conj2}
For $n\geq 2$, 
 \begin{equation*} 
E_x\{N(\mfn,\mfm)^n\}\sim \frac{4 }{\pi(n+1)}\left\{\frac{2\zeta_{\mathbb{Q}(i)}(n)}{\zeta_{\mathbb{Q}(i)}(n+1)}-1\right\}x^{n-1}. 
 \end{equation*}
\end{conj2}

Using \textsc{Matlab}, we first compiled a list of all pairs of Gaussian integers in the first quadrant with norm $x$ or less and used the Euclidean Algorithm to find all possible greatest common divisors. We determined the $n^{\text{th}}$ moment by raising the norm of each greatest common divisor to the $n^{\text{th}}$ power, summed the terms together, and then divided the result by the total number of pairs of Gaussian integers in the first quadrant with norm $x$ or less. We have graphed the results in Figures \ref{fig:2nd} - \ref{fig:5th} below for the cases when  $n=2,3,4$ and 5 with $x=50,000$. In Table \ref{tbl}, we have listed the main term of the best fit curve corresponding to each graph as compared against the conjectured main term for each value of $n$.

\begin{table}[h!] 
\caption{The main term of the $n^{\text{th}}$ moment of the norm of the greatest common divisor of pairs of Gaussian integers with norm at most $x$.}
\centering
ÊÊÊÊ\begin{tabular}{c c c} 
\hline\hline
ÊÊÊÊÊÊÊ Moment ($n$) & Numerically Derived Term & Conjectured Term
 \\ \hline 
ÊÊÊÊÊÊÊÊÊÊÊÊÊÊÊÊÊÊÊ 2 & 0.63952$x$ & 0.67364$x$Ê \\  [0.3ex] 
ÊÊÊÊÊÊÊÊÊÊÊÊÊÊÊÊÊÊÊ 3&  0.37018$x^2$ &0.37444$x^2$Ê \\  [0.3ex] 
ÊÊÊÊÊÊÊÊÊÊÊÊÊÊÊÊÊÊÊ 4 & 0.27238$x^3$ &0.27309$x^3$Ê \\  [0.3ex] 
ÊÊÊÊÊÊÊÊÊÊÊÊÊÊÊÊÊÊÊ 5  & 0.21914$x^4$ & 0.21928$x^4$ÊÊÊÊÊÊ \\  [0.3ex] 
\hline
ÊÊÊÊ\end{tabular}\label{tbl}
\end{table}

\begin{figure}[H]
\centering
\caption{The graph of $E_x\{N(\mfn,\mfm)^2\}$ for $1 \leq x \leq 50,000$. The best fit curve is 0.63952$x$+ 0.5753.} 
\includegraphics[scale=0.5]{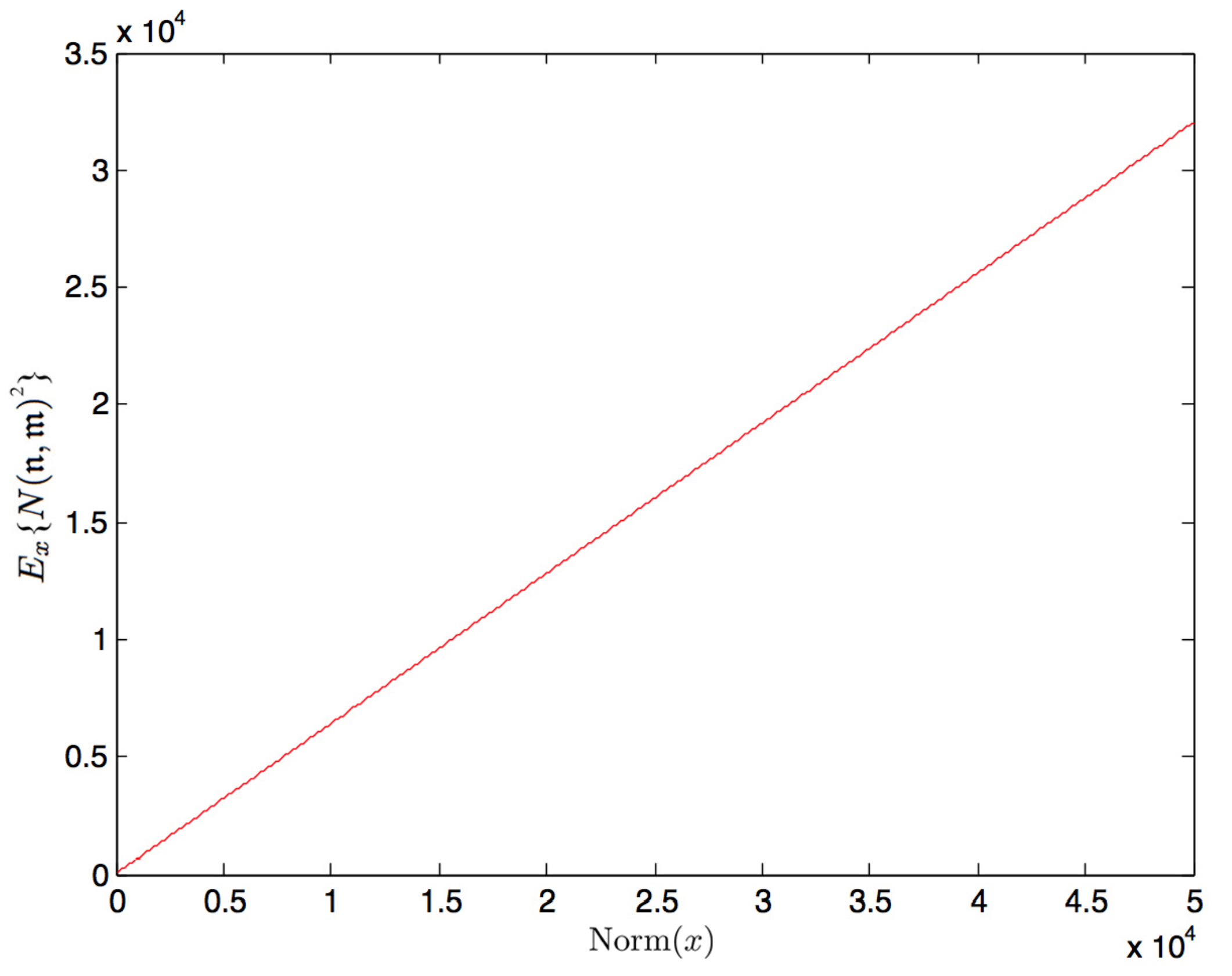}
\label{fig:2nd}
\end{figure}

\begin{figure}[H]
\centering
\caption{The graph of $E_x\{N(\mfn,\mfm)^3\}$ for $1 \leq x \leq 50,000$.  The best fit curve is 0.37018$x^2$ + 0.69337$x$ - 584.8498.} 
\includegraphics[scale=0.5]{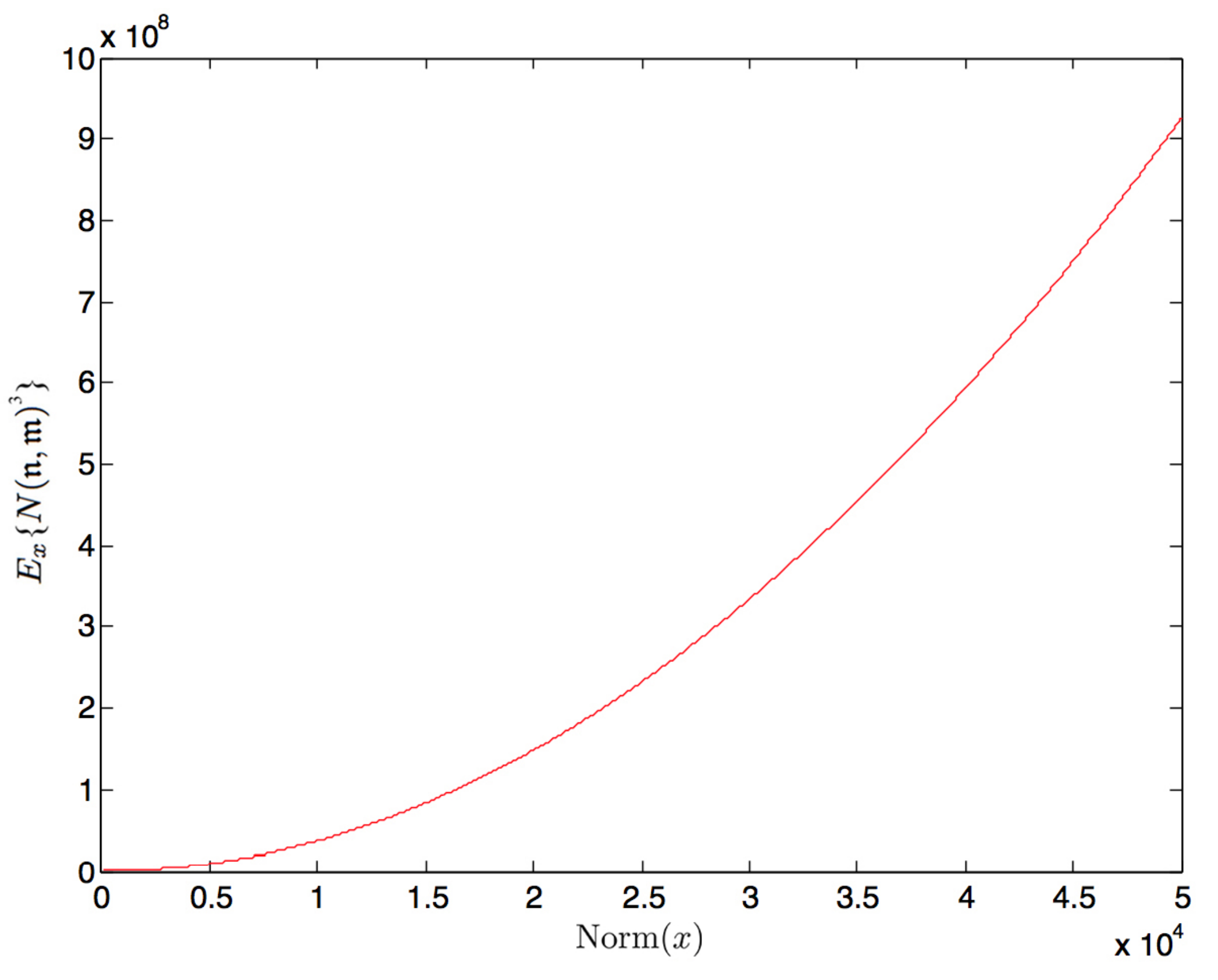}
\label{fig:3rd}
\end{figure}

\begin{figure}[H]
\centering
\caption{The graph of $E_x\{N(\mfn,\mfm)^4\}$ for $1 \leq x \leq 50,000$. The best fit curve is 0.27238$x^3$+0.80149$x^2$ - 3723.1433$x$+12324561.4508.}
\includegraphics[scale=0.5]{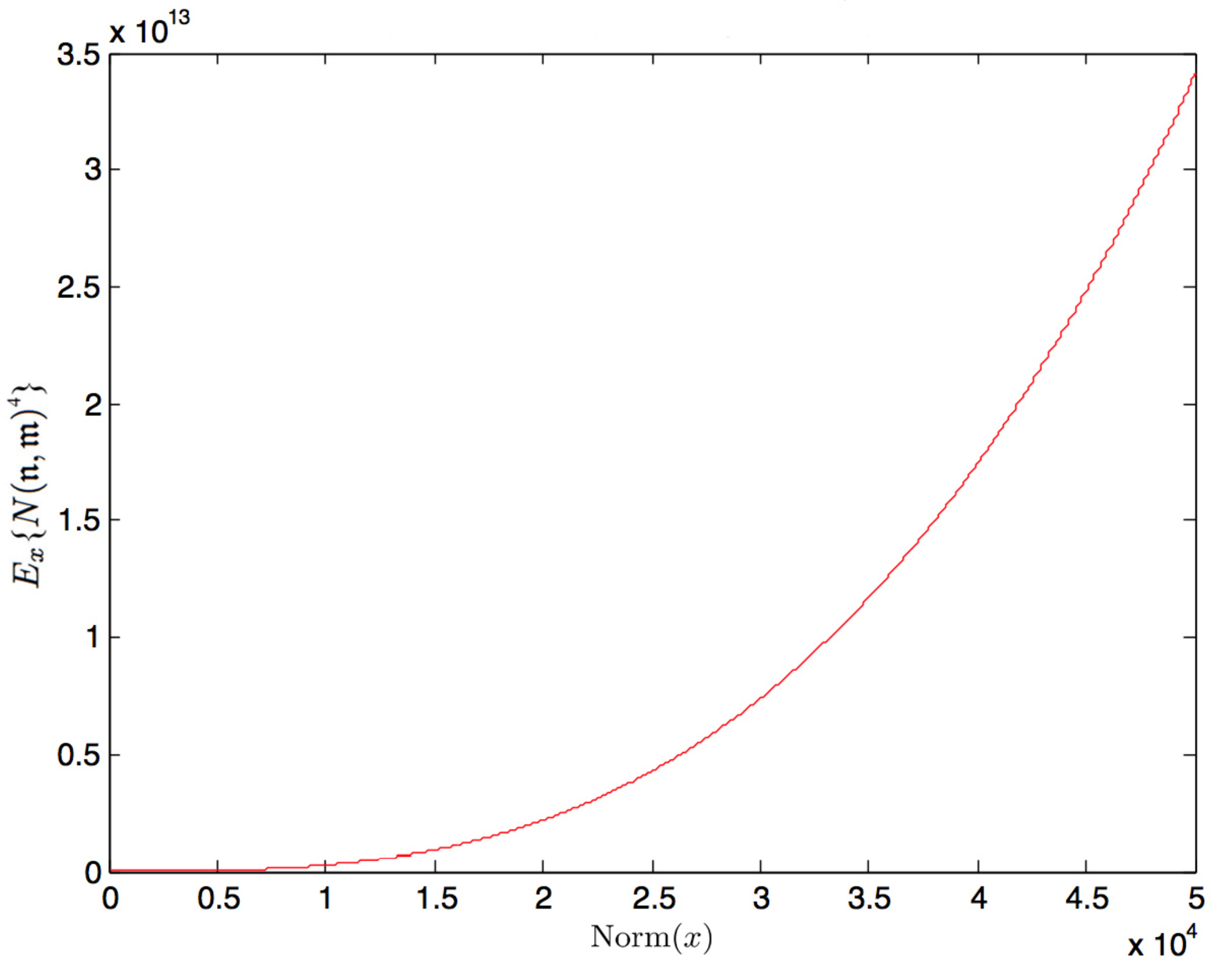}
\label{fig:4th}
\end{figure}

\begin{figure}[H]
\centering
\caption{The graph of $E_x\{N(\mfn,\mfm)^5\}$ for $1 \leq x \leq 50,000$. The best fit curve is $0.21914x^4+0.92436x^3-9773.8223x^2+92150266.2382x-190551355734.3794$.}
\includegraphics[scale=0.8]{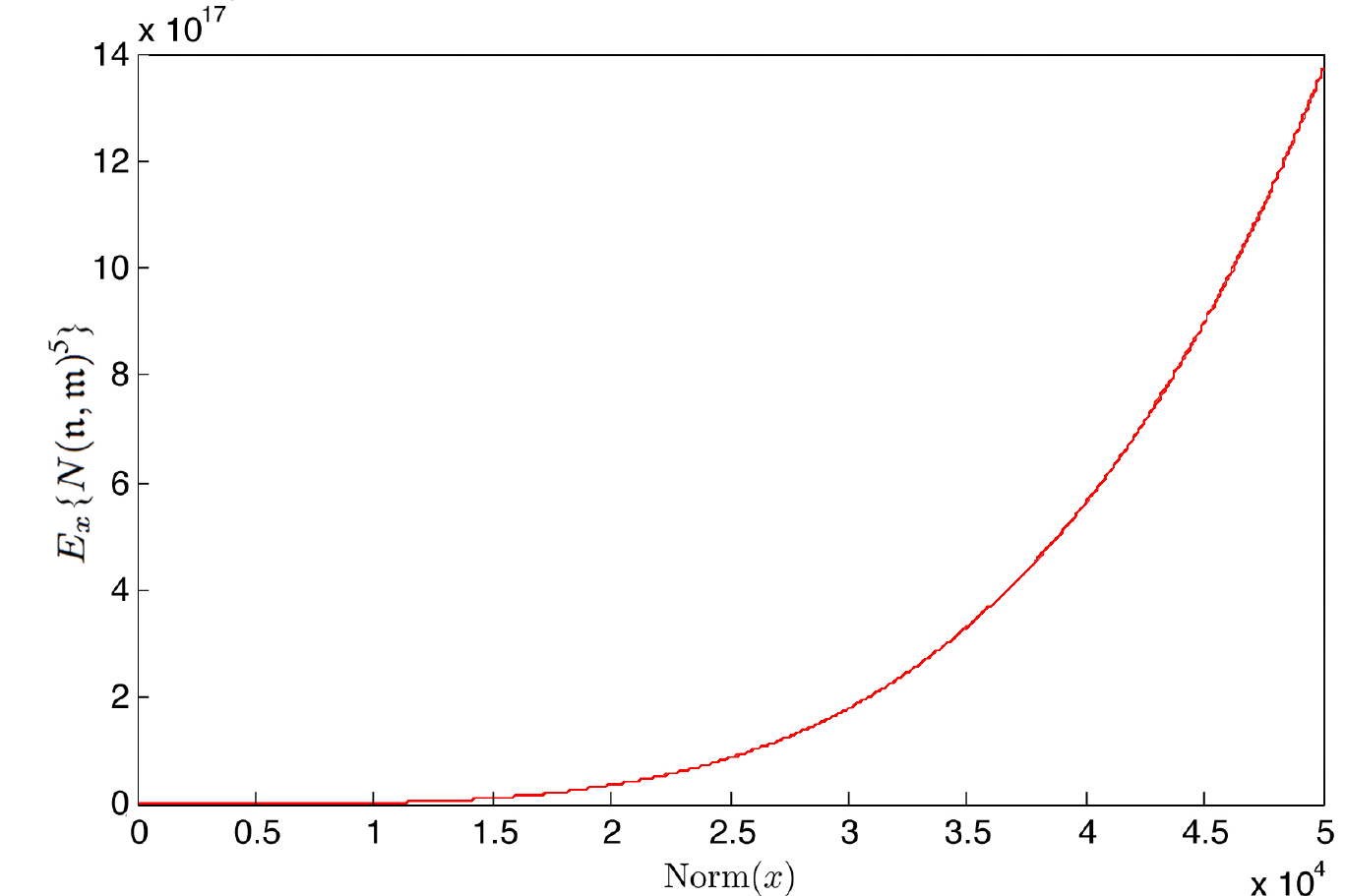}
\label{fig:5th}
\end{figure}

\section{Acknowledgements} 
The authors were supported by the  Rich Summer Internship grant of the Dr. Barnett and Jean-Hollander Rich Scholarship Fund and would like to thank both the selection committee and donors. For additional funding, the first author was supported by the NIH Maximizing Access to Research Careers (MARC) U-STAR grant [5T34 GM007639], the second author by the National Science Foundation [DMS-1201446] and the third author by the City College Fellowships Program. Gratitude also goes to Joseph Dacanay whose help with \textsc{Matlab} laid the groundwork for our experimental data. We thank Giacomo Micheli for making us aware of his work and for his comments on our paper. The authors are also grateful to the referee for helpful comments regarding the statement and proof of Theorem \ref{thm:hmom}. We especially thank our advisor, Dr. Brooke Feigon, without whose patient instruction and insightful comments this work would not have been possible. 

\end{document}